\documentclass[11pt,reqno]{amsart}
\let\Im\relax
\DeclareMathOperator{\Im}{Im}

\DeclareMathOperator{\hyp}{\mu_{{\it{X}}}^{hyp}}
\DeclareMathOperator{\hypxd}{\mu_{{\it{X^d}}}^{hyp}}
\DeclareMathOperator{\canxd}{\mu_{{\it{X^d}}}^{can}}
\DeclareMathOperator{\cansymx}{\mu_{{\it{\mathrm{Sym}^d(X)}}}^{can}}
\DeclareMathOperator{\canvol}{\mu_{{\it{X^d}},vol}^{can}}
\DeclareMathOperator{\shypvol}{\mu_{{\it{X^d}},vol}^{shyp}}
\DeclareMathOperator{\shypxd}{\mu_{{\it{X^d}}}^{shyp}}
\DeclareMathOperator{\can}{\mu_{{\it{X}}}^{can}}
\DeclareMathOperator{\shyp}{\mu^{shyp}_{{\it{X}}}}

\DeclareMathOperator{\bx}{\mathcal{B}_{{\it{X}}}}
\DeclareMathOperator{\bxn}{\mathcal{B}_{{\it{X_N}}}}
\DeclareMathOperator{\bxd}{\mathcal{B}_{{\it{X^d}}}}
\DeclareMathOperator{\xd}{{\it{X^{d}}}}
\DeclareMathOperator{\rx}{{\it{r_{_X}}}}
\DeclareMathOperator{\rxn}{{\it{r_{_{X_N}}}}}
\DeclareMathOperator{\symx}{{\it{{\mathrm{Sym}^{d}(X)}}}}

\DeclareMathOperator{\vx}{\mathrm{vol_{\mathrm{hyp}}}}

\newtheorem{thm}{Theorem}[section]
\newtheorem{lem}[thm]{Lemma}
\newtheorem{prop}[thm]{Proposition}
\newtheorem{cor}[thm]{Corollary}

\theoremstyle{definition}
 
\newtheorem{rem}[thm]{Remark}

\numberwithin{equation}{section}

\begin{document}

\title[K\"ahler metrics over $\symx$]{On the K\"ahler metrics over $\symx$}

\author[A. Aryasomayajula]{Anilatmaja Aryasomayajula}

\address{Department of Mathematics,University of Hyderabad, Prof. C.~R.~Rao Road, 
Hyderabad 500046, India}

\email{anilatmaja@gmail.com}

\author[I. Biswas]{Indranil Biswas}

\address{School of Mathematics, Tata Institute of Fundamental
Research, Homi Bhabha Road, Bombay 400005, India}

\email{indranil@math.tifr.res.in}

\author[A. S. Morye]{Archana S. Morye}

\address{Department of Mathematics,University of Hyderabad, Prof. C.~R.~Rao Road, 
Hyderabad 500046, India}

\email{sarchana.morye@gmail.com}

\author[T. Sengupta]{Tathagata Sengupta}

\address{Department of Mathematics,University of Hyderabad, Prof. C.~R.~Rao Road, 
Hyderabad 500046, India}

\email{tsengupta@gmail.com}



\date{}

\begin{abstract}
Let $X$ be a compact connected Riemann surface of genus $g$, with $g\,
\geq\, 2$. For each $d\, <\,\eta(X)$, where $\eta(X)$ is the gonality
of $X$, the symmetric product $\text{Sym}^d(X)$ embeds into
$\text{Pic}^d(X)$ by sending an effective divisor of degree $d$ to the
corresponding holomorphic line bundle. Therefore, the restriction of
the flat K\"ahler metric on $\text{Pic}^d(X)$ is a K\"ahler metric on
$\text{Sym}^d(X)$. We investigate this K\"ahler metric on
$\text{Sym}^d(X)$. In particular, we estimate it's Bergman kernel. We
also prove that any holomorphic automorphism of $\text{Sym}^d(X)$ is
an isometry.
\end{abstract}

\maketitle

\section{Introduction}

Symmetric products of Riemann surfaces were studied by Macdonald
\cite{Ma}; he explicitly computed their cohomologies. Interests on
these varieties revived when it was realized that they constitute
examples of vortex moduli spaces \cite{Br1}, \cite{Br2}, \cite{Ga}.
One of the questions was to compute the volume, which was resolved in
a series of papers \cite{Na}, \cite{MN}, \cite{Pe}; see also \cite{Ba}
for K\"ahler structure on vortex moduli spaces.

Let $X$ be a compact connected Riemann surface of genus $g$, with $g\,
\geq\, 2$, and let $\eta(X)$ denote the gonality of $X$ (this means
that $X$ admits a nonconstant holomorphic map to ${\mathbb C}{\mathbb
P}^1$ of degree $\eta(X)$ and it does not have any smaller degree
nonconstant holomorphic map to ${\mathbb C}{\mathbb P}^1$). Take
any integer $1\, \leq\, d\, <\, \eta(X)$. Let

$$
\varphi\, :\, \text{Sym}^d(X)\,\longrightarrow\, \text{Pic}^d(X)
$$ 

be the map from the symmetric product that sends any $\{x_1,\,
\ldots\, ,x_d\}$ to the holomorphic line bundle ${\mathcal
O}_X(x_1+\cdots\, +x_d)$. We prove that $\varphi$ is an embedding.

The natural inner product on $H^0(X,\, K_X)$, where
$K_X\,\longrightarrow\, X$ is the holomorphic cotangent bundle,
produces a flat K\"ahler metric on $\text{Pic}^d(X)$. It is natural to
construct a metric on $\text{Sym}^d(X)$ by pulling back the flat
metric using the embedding $\varphi$; see \cite{Ri}, \cite{MR} (especially
\cite[p.~1137, (1.2)]{MR}, \cite[\S~7]{MR}). Our aim here is
to study this metric on $\text{Sym}^d(X)$. We prove that any
holomorphic automorphism of $\text{Sym}^d(X)$ is in fact an
isometry. Our main result is estimation of the Bergman kernel of the
metric.

Classically, the Bergman kernel which is the reproducing kernel for
$L^2$-holomorphic functions has been extensively studied in complex
analysis. The generalization of the Bergman kernel to complex
manifolds as the kernel for the projection onto the space of harmonic
$(p, q)$-forms with $L^2$-coefficients carries the information on the
algebraic and geometric structures of the underlying manifolds.

Using results from \cite{jkf} and \cite{jl}, we derive the following
estimate for $\bx(z)$, the Bergman kernel associated to the Riemann
surface $X$:
$$
\bx(z)\leq \frac{48}{\pi}+\frac{4}{3\pi\sinh^{2}(\rx\slash 4)}, 
$$
where $\rx$ denotes the injectivity radius of $X$.

We also study the above estimate for admissible sequences of compact
hyperbolic Riemann surfaces. Our estimates are optimal, and these
estimates continue to hold true for any compact hyperbolic Riemann surface.

\section{Comparison of K\"ahler metrics}

In this section, we introduce the hyperbolic and canonical metrics
defined on a compact hyperbolic Riemann surface. Furthermore, we
introduce the Bergman kernel, and derive estimates for it. We then
extend these estimates to admissible sequences of compact hyperbolic
Riemann surfaces.

\subsection{Canonical and hyperbolic metrics}\label{se2.1}

Let $X$ be a compact, connected Riemann surface of genus $g$, with
$g\,>\, 1$. Let
\begin{align*}
\mathbb{H}\,:=\,\lbrace z\,=\,x+\sqrt{-1}y\,\in\, \mathbb{C}\,\mid\,y\,>\,0 \rbrace 
\end{align*}
be the upper half-plane. Using the uniformization theorem $X$ can be
realized as the quotient space $\Gamma\backslash\mathbb{H}$, where
$\Gamma\,\subset\, \mathrm{PSL}_{2}(\mathbb{R})$ is a torsionfree
cocompact Fuchsian subgroup acting on $\mathbb{H}$, via fractional
linear transformations.

Locally, we identify $X$ with its universal cover $\mathbb{H}$ using
the covering map $\mathbb{H}\,\longrightarrow\, X$.

The holomorphic cotangent bundle on $X$ will be
denoted by $K_X$. Let $$\mathrm{Jac}(X)\,=\,\text{Pic}^{0}(X)$$ be the
Jacobian variety that parametrizes all the (holomorphic) isomorphism classes of topologically trivial
holomorphic line bundles on $X$. It is equipped with a flat K\"ahler metric
$g_J$ given by the Hermitian structure on $H^0(X,\, K_X)$ defined by
\begin{equation}\label{m0}
(\alpha\, , \beta)\, \longmapsto\, \frac{\sqrt{-1}}{2}\int_X \alpha\wedge \overline{\beta}\, .
\end{equation}
Fix a base point $x_0\, \in\, X$. Let
$$
{\rm AJ}_X\, :\, X\, \longrightarrow\, \mathrm{Jac}(X)
$$ 

be the Abel-Jacobi map that sends any $x\,\in\, X$ to the holomorphic
line bundle on $X$ of degree zero given by the divisor $x-x_0$. It is
a holomorphic embedding of $X$. The pulled back K\"ahler metric ${\rm
AJ}^*_X g_J$ on $X$ is called the {\it canonical metric}. The
$(1,1)$-form on $X$ associated to the canonical metric is denoted by
$\can$.

The canonical metric has the following alternate description. Let
$S_{2}(\Gamma)$ denote the $\mathbb{C}$-vector space of cusp forms of
weight-$2$ with respect to $\Gamma$. Let $\lbrace f_{1}\, ,\ldots\,
,f_{g}\rbrace $ denote an orthonormal basis of $S_{2}(\Gamma)$ with
respect to the Petersson inner product. Then, the $(1,1)$-form $\can(z)$
corresponding to the canonical metric of $X$ is given by
\begin{equation}\label{candefn1}
\can(z)\,:=\, \frac{\sqrt{-1}}{2g} \sum_{j=1}^{g}\left|f_{j}(z)\right|^{2}dz\wedge
d\overline{z}\, .
\end{equation}
The volume of $X$ with respect to the canonical metric is one.

Consider the hyperbolic metric of $X$, which is compatible with the
complex structure on $X$ and has constant negative curvature $-1$. We
denote by $\hyp$ the $(1,1)$--form on $X$ corresponding to it. The
hyperbolic form on $\mathbb H$ is given by

$$
\frac{\sqrt{-1}}{2}\cdot\frac{dz\wedge d\overline{z}}{{\Im(z)}^{2}}\, .
$$

So on $X$, the form $\hyp(z)$ is given by

$$
 \hyp(z)\,:=\, 
\frac{\sqrt{-1}}{2}\cdot\frac{dz\wedge d\overline{z}}{{\Im(z)}^{2}},
$$

for $z\in X$. The total volume $\vx(X)$ of $X$ with respect to the
hyperbolic metric $\hyp$ is given by the formula
\begin{equation*}
\vx(X) \,=\, 4\pi\big(g -1 \big)\, . 
\end{equation*}
Let
\begin{equation*}
 \shyp(z)\,:= \,\frac{\hyp(z)}{ \vx(X)}
\end{equation*}
denote the rescaled hyperbolic metric on $X$, which is normalized in such
a way that the volume of $X$ is one.

\subsection{Estimates of the Bergman kernel}\label{subsec2.2}

For any $z\,\in\, X$, the Bergman kernel $\bx$ associated to the
Riemann surface $X$ is given by the following formula
\begin{align*}
\bx(z)\,:=\,\sum_{j=1}^{g}y^{2}\left|f_{j}(z)\right|^{2}\, ,
\end{align*}
where $y\,=\, {\rm Im}\, z$.

The injectivity radius $\rx$ of $X$ is defined as 
\begin{align*}
 \rx\,:=\, \inf\big\lbrace{d_{\mathbb{H}}(z,\gamma z)\,\mid\, z\,\in\, \mathbb{H},
\, \gamma\in\Gamma\backslash\lbrace \mathrm{id}\rbrace \big\rbrace}\, ,
\end{align*}
where $d_{\mathbb{H}}(z,\gamma z)$ denotes the hyperbolic distance
between $z$ and $\gamma z$.

\vspace{0.2cm}

\noindent Let $f$ be any positive, smooth, real valued decreasing function defined
on $\mathbb{R}_{\geq 0}$. From \cite[Lemma 4]{jl}, for any
$\delta\,>\, \rx\slash 2$, and assuming that all the involved integrals
exist, we have the following inequality

$$
\int_{0}^{\infty}f(\rho)dN_{\Gamma}(z_{1},z_{2};\rho)\,\leq\,
\int_{0}^{\delta}f(\rho)dN_{\Gamma}(z_{1},z_{2};\rho)
$$

\begin{equation}\label{jlinequality}
+\, f(\delta)\frac{\sinh(\rx\slash
 2)\sinh(\delta)}{\sinh^{2}(\rx\slash 4)} +
\frac{1}{2\sinh^{2}(\rx\slash 4)}\int_{\delta}^{\infty}f(\rho)
\sinh(\rho+\rx\slash 2)d\rho\, ,
\end{equation}

where
\begin{align*}
N_{\Gamma}(z_{1},z_{2};\rho)\,:=\,\mathrm{card}\,\big\lbrace \gamma\,\mid\,\gamma\in\Gamma,\,d_{\mathbb{H}}(z_1,
\gamma z_2)\leq \rho \big\rbrace.
\end{align*}
Notice that the above injectivity radius $\rx$ is twice the injectivity radius
defined in \cite{jl}. 

\begin{thm}\label{boundbk}
For any $z\in X$, the following estimate holds:
$$
\bx(z)\,\leq\, B_X\, :=\, \frac{48}{\pi}+\frac{4}{3\pi\sinh^{2}(\rx\slash 4)}\, . 
$$
\end{thm}

\begin{proof}
Substituting $k\,=\,1$ in inequality (13) of \cite{jkf}, we arrive at
\begin{align}\label{boundbkeqn1}
\bx(z)\,\leq\, \frac{\sqrt{2}}{3\pi}\sum_{\gamma\in\Gamma}\frac{1}{\cosh^2(\rho_{\gamma,z}\slash 2)}\int_{\rho_{\gamma,z}}^{\infty}
\frac{ue^{-u\slash 2}}{\sqrt{\cosh(u)-\cosh(\rho_{\gamma,z})}}du\, ,
\end{align}
where $\rho_{\gamma,z}\,=\,d_{\mathbb{H}}(z,\gamma z)$. Using the fact that
$u\,\leq\,\sinh(u)$ for all $u\,\geq\, 0$, 
\begin{align}
\int_{\rho_{\gamma,z}}^{\infty}
\frac{ue^{-u\slash 2}}{\sqrt{\cosh(u)-\cosh(\rho_{\gamma,z})}}du \,\leq\,
\int_{\rho_{\gamma,z}}^{\infty}\frac{ue^{-u\slash 2}}{\sqrt{\cosh(u)-1}}du\notag\\
=\,\int_{\rho_{\gamma,z}}^{\infty}\frac{ue^{-u\slash 2}}{\sqrt{2\sinh^{2}(u\slash2)}}du
\,\leq\, 
\sqrt{2}\int_{\rho_{\gamma,z}}^{\infty}e^{-u\slash 2}du\,=\,
2\sqrt{2}e^{-\rho_{\gamma,z}}\, .\label{boundbkeqn2}
\end{align}
Combining \eqref{boundbkeqn1} and \eqref{boundbkeqn2}, and using the fact that the
inequality $\cosh(u)\,\geq\, e^{u}\slash 2$ holds for all $u\,\geq\, 0$, it follows that 
\begin{align*}
 \bx(z)\,\leq\, \frac{4}{3\pi}\sum_{\gamma\in\Gamma}
\frac{e^{-\rho_{\gamma,z}}}{\cosh^2(\rho_{\gamma,z}
 \slash 2)}\,\leq\, \frac{16}{3\pi}\sum_{\gamma\in\Gamma}
\frac{e^{-\rho_{\gamma,z}}}{e^{\rho_{\gamma,z}}}\,=\,\frac{16}{3\pi}\int_{0}^{\infty}
 e^{-2\rho}dN_{\Gamma}(z,\gamma z;\rho)\, .
\end{align*}
As $e^{-2\rho}$ is a monotonically decreasing function in $\rho\,\in\,
\mathbb{R}_{\geq 0}$, using \eqref{jlinequality} we compute
$$
\bx(z)\,\,\,\leq\,\, \,
\frac{16}{3\pi}\int_{0}^{\frac{3\rx}{4}}e^{-2\rho}dN_{\Gamma}(z,\gamma z;\rho)
$$
\begin{equation}\label{boundbkeqn3}
+
\frac{16e^{-\frac{3\rx}{2}}\sinh(\rx\slash 2)\sinh(3\rx\slash 4)}{3\pi\sinh^{2}(\rx\slash 4)}
+\frac{8}{3\pi\sinh^{2}(\rx\slash 4)}\int_{\frac{3\rx}{4}}^{\infty}
e^{-2\rho}\sinh\bigl(\rho+\frac{\rx}{2}\bigr)d\rho\,.
\end{equation}

From the definition of the injectivity radius $\rx$ we have 
\begin{equation}\label{boundbkeqn4}
\frac{16}{3\pi}\int_{0}^{\frac{3\rx}{4}}e^{-2\rho}dN_{\Gamma}(z,\gamma z;\rho)
\,=\, \frac{16}{3\pi}\,. 
\end{equation}
Using the fact that $\sinh(u)$ is a monotone increasing function and that
the inequality $\cosh(u)\,\leq \, e^{u}$ holds for all $u\,\geq\, 0$, we have the
following estimate for the second term on the right-hand side of inequality in
\eqref{boundbkeqn3}:
$$
\frac{16e^{-\frac{3\rx}{2}}\sinh(\rx\slash 2)\sinh(3\rx\slash 4)}{3\pi\sinh^{2}(\rx\slash 4)}
\,\leq\, 
\frac{16e^{-\frac{3\rx}{2}}\sinh(\rx\slash 2)\sinh(\rx)}{3\pi\sinh^{2}(\rx\slash 4)}
$$

\begin{equation}\label{boundbkeqn5}
\leq\,\frac{128e^{-\frac{3\rx}{2}}\cosh^2(\rx\slash 4)\cosh(\rx\slash 2)}{3\pi}\,
\leq \,\frac{128e^{-\frac{\rx}{2}}}{3\pi}\,\leq\, \frac{128}{3\pi}\,. 
\end{equation}

Using the fact that $$\sinh(u)\,\leq\, e^{u}\slash 2$$ for all $u\,\geq\, 0$,
we derive the following estimate for 
the third term on the right-hand side of the inequality in \eqref{boundbkeqn3}:

$$
\frac{8}{3\pi\sinh^{2}(\rx\slash 4)}\int_{\frac{3\rx}{4}}^{\infty}e^{-2\rho}
\sinh\bigl(\rho+\frac{\rx}{2}\bigr)d\rho
$$

\begin{equation}\label{boundbkeqn6}
\leq\,
\frac{4e^{\frac{\rx}{2}}}{3\pi\sinh^{2}(\rx\slash 4)}\int_{\frac{3\rx}{4}}^{\infty}e^{-\rho}d\rho
\,=\,\frac{4e^{-\frac{\rx}{4}}}{3\pi\sinh^{2}(\rx\slash 4)}
\,\leq\, \frac{4}{3\pi\sinh^{2}(\rx\slash 4)}\, .
\end{equation}

Now the theorem follows from \eqref{boundbkeqn4}, \eqref{boundbkeqn5}, and \eqref{boundbkeqn6}.
\end{proof}

Let $\lbrace X_{N}\rbrace_{N\in\mathcal{N}}$, indexed by
$\mathcal{N}\,\subseteq\, \mathbb{N}$, be a set of compact hyperbolic
Riemann surfaces. We say that the sequence is \emph{admissible} if it
is one of the following two types:
\begin{enumerate}
\item If $\mathcal{N}\,=\,\mathbb{N}$ and $N\,\in\,\mathcal{N}$, then
$X_{N+1}$ is a finite degree unramified cover of $X_{N}$.

\item Let $\mathcal{N} \,\subset\, \mathbb{N}$ be such that for each
$N \,\in \, \mathcal{N}$, the modular curves $X_{0}(N)$, $X_{1}(N)$,
$X(N)$, have genus $g\,>\,1$. We consider families of modular curves
$\lbrace X_{N} \rbrace_{N \,\in \,\mathcal{N}}$ given by
$$
\lbrace X_{0}(N) \rbrace_{N \in \mathcal{N}},\,\,\, \lbrace X_{1}(N) \rbrace_{N
\,\in\, \mathcal{N}},\,\,\, \lbrace X(N) \rbrace_{N \in \mathcal{N}}\, .
$$
\end{enumerate}
See \cite[p. 695--696, Definition 5.1]{jkcomp}.

Let $q_{_{\mathcal{N}}}\,\in\,\mathcal{N}$ be the minimal element of
the indexing set $\mathcal{N}$. So in Case (1), we have
$q_{_{\mathcal{N}}}\,=\,0$, while in Case (2), the integer
$q_{_{\mathcal{N}}}$ is the smallest prime in $\mathcal{N}$.

\begin{cor}\label{boundbkcor}
Let $\lbrace X_{N}\rbrace_{N\in\mathcal{N}}$ be an admissible sequence
of compact hyperbolic Riemann surfaces. Then, for all
$N\,\in\,\mathcal{N}$, the Bergman kernel $\bxn(z)$ is bounded by a
constant which depends only on the Riemann surface
$X_{q_{_{\mathcal{N}}}}$.
\end{cor}

\begin{proof}
From Theorem \ref{boundbk}, we have
\begin{align}\label{corxeqn1}
\bxn(z)\, \leq\,B_{X_{N}}\,= O\left(\frac{1}{r^{2}_{X_{N}}}\right)\,.
\end{align}
Recall that injectivity radius $\rxn$ is equal to $\ell_{X_{N}}$, the
length of the shortest geodesic on $X_N$. From assertion (a) in
\cite[Lemma 5.3]{jkcomp} we know that for all $N\,\in\,\mathcal{N}$,
the number $\frac{1}{\rxn}$ is bounded by a number that depends only
on the Riemann surface $X_{q_{_{\mathcal{N}}}}$. Therefore, the
estimate \eqref{corxeqn1} completes the proof.
\end{proof}

\begin{rem}
In \cite{bergmanbounds}, B.-Y.~Chen and S.~Fu have also derived a
similar estimate for the Bergman kernel as in Corollary
\ref{boundbkcor}. However, their estimate is valid only for any
compact hyperbolic Riemann surfaces with injectivity radius greater
than or equal to $\log(3)$.
\end{rem}

\section{Cartesian product $\xd$}

In this section, we introduce the hyperbolic and canonical metrics defined over
the $d$-fold Cartesian product $\xd$ of $X$. We, then 
compute an estimate for the volume form associated to the canonical metric.

\subsection{Canonical and hyperbolic metrics}\label{se3.1}

Take $X$ as before. Let $\xd$ denote the $d$-fold Cartesian 
product $X\times \cdots\times X$. For each $1\, \leq\, i\,\leq\, d$, let
$$
p_i\, :\, X^d\,\longrightarrow\, X
$$
be the projection to the $i$-th factor. Define
$$
\hypxd\,=\, \sum_{i=1}^d p^*_i \hyp \ \ 
\text{ and }\ \ \shypxd \,=\, \sum_{i=1}^d p^*_i \shyp\, .
$$ 

We denote by $\shypvol$ the volume form associated to $\shypxd$. Note
that the total volume of $X^d$ with respect to $\shypvol$ is $1$,
because the total volume of $X$ with respect to $\shyp$ is $1$.

With respect to a local coordinate $z\,=\,(z_1\, ,\ldots\, ,z_d)$ on
$X^d$, where $z_i\,=\, x_i+\sqrt{-1}y_i$ are hyperbolic coordinates on
$X$, the hyperbolic volume form is given by
\begin{align*}
\shypvol(z)\,=\,\frac{1}{(\vx(X))^{d}}\bigwedge_{j=1}^{d}\frac{\sqrt{-1}}{2}\cdot
\frac{dz_{j}\wedge d\overline{z}_{j}}{y_{j}^{2}}\,=\,
\frac{1}{(4\pi(g-1))^{d}}\bigwedge_{j=1}^{d}\frac{\sqrt{-1}}{2}\cdot
\frac{dz_{j}\wedge d\overline{z}_{j}}{y_{j}^{2}}\, .
\end{align*}

The gonality of $X$ is defined to be the smallest among all positive
integers $m$ such that $X$ admits a nonconstant holomorphic map to
${\mathbb C}{\mathbb P}^1$ of degree $m$. The gonality of $X$ will be
denoted by $\eta(X)$. So $\eta(X)\,=\,2$ if and only if $X$ is
hyperelliptic.

We assume that $d\, <\, \eta(X)$.

Let $\text{Pic}^d(X)$ denote the component of the Picard group of $X$
that parametrizes all the holomorphic line bundles of degree
$d$. Consider the holomorphic map
\begin{equation}\label{phi}
\phi\, :\, X^d\, \longrightarrow\, \mathrm{Pic}^{d}(X)\, ,\ \ (x_1\, ,\ldots\, , x_d)\,
\longmapsto\, {\mathcal O}_X(x_1+\cdots +x_d)\, .
\end{equation}
Since $d\, <\, m$, it can be shown that the fibers of the above map
$\phi$ are zero dimensional. Indeed, if $$\phi((x_1, \ldots,x_d))\,=\,
\phi((y_1, \ldots,y_d))\, ,$$ the holomorphic line bundle ${\mathcal
O}_X(x_1+\cdots +x_d)$ has two nonzero sections given by the two
effective divisors $x_1+\cdots +x_d$ and $y_1+\cdots +y_d$. These two
sections can't be linearly independent because that would contradict
the assumption on $d$ that it is strictly smaller than
$\eta(X)$. Since two sections are constant multiples of each other, it
follows that $(x_1, \ldots,x_d)$ and $(y_1, \ldots,y_d)$ differ by a
permutation of $\{1\, ,\ldots\, ,d\}$. Therefore, we have the
following:

\begin{lem}\label{l1}
Any two points of $X^d$ lying in a fiber
of the map $\phi$ differ by a permutation of $\{1\, ,\ldots\, ,d\}$.
\end{lem}

The variety $\text{Pic}^d(X)$ is a torsor for $\mathrm{Jac}(X)$, because any two
holomorphic line bundles of degree $d$ differ by tensoring with a unique holomorphic
line bundle of degree zero. Therefore, by fixing a point of $\text{Pic}^d(X)$ we may
identify $\mathrm{Jac}(X)$ with $\text{Pic}^d(X)$. Using this identification,
we get a K\"ahler metric on $\text{Pic}^d(X)$ given by the metric on $\mathrm{Jac}(X)$
constructed in \eqref{m0}. This metric on $\text{Pic}^d(X)$ will be denoted by
$g_d$. We note that $g_d$ does not depend on the choice of the point in
$\text{Pic}^d(X)$ used in identifying $\mathrm{Jac}(X)$ with $\text{Pic}^d(X)$.

The pullback $\phi^*g_d$ is the canonical metric on $X^d$, which we denote by $\canxd$.
The canonical metric degenerates along the divisor where two or more coordinates coincide
(where the action of the group of permutations of $\{1\, ,\ldots\, ,d\}$ is not free). In
Remark \ref{re1} we will see that this is precisely the locus where $\canxd$ degenerates.

As in Section \ref{se2.1}, let $\lbrace f_{1}\, ,\ldots\, ,f_{g}\rbrace $ be an
orthonormal basis of $S_{2}(\Gamma)$ with respect to the Petersson 
inner product. The (1,1)-form associated to the canonical metric $\canxd$ is given by
\begin{align}\label{eqn2}
\canxd\,=\,\frac{\sqrt{-1}}{2g^d}\sum_{j=1}^{g}\sum_{a,b=1}^{d}f_{j}(z_{a})
\overline{f_{j}(z_{b})}dz_{a}\wedge d\overline{z}_{b}\, .
\end{align}
The volume form associated to the canonical metric $\canxd$ measures the total
volume of $\xd$ to be one. 

For any $z=(z_{1}\, ,\ldots\, ,z_{d})\,\in\,\xd$, the Bergman kernel associated to $\xd$ is 
given by the formula
\begin{align*}
\bxd(z)\,=\, \prod_{i=1}^{d}\bx(z_i,w_i)\, . 
\end{align*}

\subsection{Estimates of $\canvol$}

In this subsection, using the estimate for the Bergman kernel $\bx(z)$ derived in Theorem 
\ref{boundbk}, we estimate $\canvol$, the volume form associated to the canonical metric 
$\canxd$.

\begin{thm}\label{boundvol}
For any $z\,\in\, \xd$, the following inequality holds:
\begin{align*}
\Bigg|\frac{\canvol(z)}{\shypvol(z)}\Bigg| \,\leq\,
(d!)^2\bigg(\frac{\vx(X)B_X}{g^{d-1}}\bigg)^d\, .
\end{align*}
\end{thm}

\begin{proof}
For any $z\,=\,(z_1\, ,\ldots\, ,z_d)\,\in\,\xd$, the canonical volume form $\canvol$ is
given by 
$$
 \canvol(z)\,=\,
$$
$$
\Bigg(\frac{\sqrt{-1}}{2g^d}\Bigg)^d\sum_{\substack{j_1,\ldots ,j_d
\in\lbrace1,\ldots,g \rbrace\\
 \sigma,\tau\in S_d}}f_{j_1}(z_{\sigma(1)})\overline{f_{j_1}(z_{\tau(1)})}\cdots
 f_{j_d}(z_{\sigma(d)})\overline{f_{j_d}(z_{\tau(d)})}\bigwedge_{k=1}^{d}dz_{\sigma(k)}
\wedge d\overline{z}_{\tau(k)}\,=
$$
$$
\Bigg(\frac{\sqrt{-1}}{2g^d}\Bigg)^d
\sum_{\substack{j_1,\ldots ,j_d\in\lbrace1,\ldots,g
 \rbrace\\
\sigma,\tau\in S_d}} \mathrm{sgn}(\sigma)\mathrm{sgn}(\tau)f_{j_1}(z_{\sigma(1)})
 \overline{f_{j_1}(z_{\tau(1)})}\cdots f_{j_d}(z_{\sigma(d)})\overline{f_{j_d}(z_{\tau(d)})}
 \bigwedge_{k=1}^{d}dz_{k}\wedge d\overline{z}_{k}\, .
$$
Using the above expression, we observe that
$$
\Bigg|\frac{\canvol(z)}{\shypvol(z)}\Bigg|^2 =\Bigg(\frac{\vx(X)}{g^d}\Bigg)^{2d}
$$
$$
\times
\Bigg|\bigg(\prod_{k=1}^{d}y_{k}^2\bigg)\cdot\sum_{\substack{j_1,\ldots ,j_d\in\lbrace1,\ldots,g
 \rbrace\\ \sigma,\tau\in S_d}} \mathrm{sgn}(\sigma)\mathrm{sgn}(\tau)f_{j_1}(z_{\sigma(1)})
\overline{f_{j_1}(z_{\tau(1)})}\cdots f_{j_d}(z_{\sigma(d)})
\overline{f_{j_d}(z_{\tau(d)})} \Bigg|^{2}\, .
$$
Since the number of terms in the above summation are $(d!)^2g^d$, we arrive at the inequality
$$
\Bigg|\frac{\canvol(z)}{\shypvol(z)}\Bigg|^2\leq (d!)^4\Bigg(\frac{g\vx(X)}{g^{d}}\Bigg)^{2d}
$$
\begin{equation}\label{boundvoleqn1}
\times
\sup_{\substack{j_1,\ldots ,j_d\in\lbrace1,\ldots,g\rbrace\\ \sigma,\tau\in S_d,\,z\in\xd}}
 \Bigg|\bigg(\prod_{k=1}^{d}y_{k}^2\bigg)\cdot f_{j_1}(z_{\sigma(1)})
\overline{f_{j_1}(z_{\tau(1)})}\cdots f_{j_d}(z_{\sigma(d)})
 \overline{f_{j_d}(z_{\tau(d)})} \Bigg|^2\, .
\end{equation}
{}From Theorem \ref{boundbk}, we derive
$$
\sup_{\substack{j_1,\ldots ,j_d\in\lbrace1,\ldots,g\rbrace\\ \sigma,\tau\in S_d,\,z\in\xd}}
\Bigg|\bigg(
\prod_{k=1}^{d}y_{k}^2\bigg)\cdot f_{j_1}(z_{\sigma(1)})
\overline{f_{j_1}(z_{\tau(1)})}\cdots 
f_{j_d}(z_{\sigma(d)})\overline{f_{j_d}(z_{\tau(d)})} \Bigg|^2
$$
\begin{equation}\label{boundvoleqn2}
\leq\, \sup_{z\in\xd}\big(\bxd(z)\big)^2\leq \big(B_{X}\big)^{2d}\, .
\end{equation}
Combining the inequalities \eqref{boundvoleqn1} and \eqref{boundvoleqn2}, the proof is completed.
\end{proof}

\section{Singularities of the canonical metric on the symmetric product}

As before, take $d\, <\,\eta(X)$.
Let $S_d$ denote the group permutation of $\{1\, ,\ldots\, ,d\}$. It acts on $X^d$ by 
permuting the factors. Let $\symx$ denote the $d$-fold symmetric product of $X$. In other 
words, $\symx$ is the quotient of $X^d$ for the action of $S_d$.

The metric $\cansymx$ on $X^d$ is clearly invariant under the action of the group
$S_d$. Let us denote the push-forward of the canonical metric $\canxd$ onto $\symx$.

\begin{prop}\label{symmsing}
Consider the map $\phi\, :\, X^d\, \longrightarrow\, \mathrm{Pic}^{d}(X)$ in \eqref{phi}.
It factors through the quotient $X^d\, \longrightarrow\, X^d/S_d \,=\, {\rm Sym}^d(X)$.
The resulting map
$$
{\rm Sym}^d(X)\, \longrightarrow\, \mathrm{Pic}^{d}(X)
$$
is an embedding.
\end{prop}

\begin{proof}
If two elements $(x_1, \ldots,x_d)$ and $(y_1, \ldots,y_d)$ of $X^d$ lie in the same orbit
for the action of $S_d$ on $X^d$, then the line bundles
${\mathcal O}_X(x_1+\cdots +x_d)$ and ${\mathcal O}_X(y_1+\cdots +y_d)$ are isomorphic.
Hence $\phi$ descends to a morphism
\begin{equation}\label{vp}
\varphi\,:\, \symx \, \longrightarrow\, \mathrm{Pic}^{d}(X)\, .
\end{equation}
{}From Lemma \ref{l1} we know that $\varphi$ is injective. Therefore, it suffices
to show that $\varphi$ is an immersion.

Take any point $\underline{x}\, =\, \{x_1\, ,\ldots\, ,x_d\}\,\in\, \symx$.
The divisor $\sum_{i=1}^d x_i$ will be denoted by $D$. Let
$$
0\, \longrightarrow\, {\mathcal O}_X(-D)\, \longrightarrow\, {\mathcal O}_X
\, \longrightarrow\, Q'(\underline{x})\,:=\, {\mathcal O}_X/{\mathcal O}_X(-D)
\, \longrightarrow\, 0
$$
be the short exact sequence corresponding to the point $\underline{x}$.
Tensoring it with the line bundle ${\mathcal O}_X(-D)^*\,=\, {\mathcal O}_X(D)$ we get
the short exact sequence
$$
0\, \longrightarrow\, End({\mathcal O}_X(-D))\,=\,{\mathcal O}_X
\, \longrightarrow\, Hom({\mathcal O}_X(-D)\, ,{\mathcal O}_X)\,=\,
{\mathcal O}_X(D)
$$
$$
\longrightarrow\, Q(\underline{x})\,:=\,
Hom({\mathcal O}_X(-D)\, ,Q'(\underline{x}))\, \longrightarrow\, 0\, .
$$
Let
\begin{equation}\label{e1}
0\, \longrightarrow\, H^0(X,\, {\mathcal O}_X)\,\stackrel{\alpha}{\longrightarrow}\,
H^0(X,\, {\mathcal O}_X(D))\,\stackrel{\beta}{\longrightarrow}\, H^0(X,\,
Q(\underline{x}))\,\stackrel{\gamma}{\longrightarrow}\, H^1(X,\, {\mathcal O}_X)
\end{equation}
be the long exact sequence of cohomologies associated to this short exact sequence
of sheaves.

The holomorphic tangent space to $\symx$ at $\underline{x}$ is
$$
T_{\underline{x}}\symx\,=\, H^0(X,\, Q(\underline{x}))\, ,
$$
and the tangent bundle of $\text{Pic}^d(X)$ is the trivial vector bundle
with fiber $H^1(X,\, {\mathcal O}_X)$. The differential at $\underline{x}$
of the map $\varphi$ in \eqref{vp}
$$
(d\varphi)(\underline{x})\, :\, T_{\underline{x}}\symx\,=\, H^0(X,\,
Q(\underline{x}))\,\longrightarrow\, T_{\varphi(\underline{x})}\text{Pic}^d(X)\,=\,
H^1(X,\, {\mathcal O}_X)
$$
satisfies the identity
\begin{equation}\label{e2}
(d\varphi)(\underline{x})\,=\, \gamma\, ,
\end{equation}
where $\gamma$ is the homomorphism in \eqref{e1}.

Now, $H^0(X,\, {\mathcal O}_X)\,=\, \mathbb C$. In the proof of Lemma \ref{l1} we saw
that
$$
H^0(X,\, {\mathcal O}_X(D))\,=\, {\mathbb C}\, .
$$
Hence the homomorphism $\alpha$ in \eqref{e1} is an isomorphism. Consequently, $\beta$ in
the exact sequence \eqref{e1} is the zero homomorphism and $\gamma$ in \eqref{e1}
is injective.

Since $\gamma$ in \eqref{e1} is injective, from \eqref{e2} we conclude that
$\varphi$ is an immersion.
\end{proof}

\begin{rem}\label{re1}
Since $\varphi$ is an embedding, the metric $\cansymx$ on 
$\symx$ is nonsingular.
Therefore, the metric $\canxd$ on $X^d$ is singular exactly on the divisor where the
quotient map $X^d\, \longrightarrow\, \symx$ is ramified. We note that this ramification
divisor consists of all points of $X^d$ such that the $d$ points of $X$ are not distinct.
\end{rem}

\section{Automorphisms of $\text{Sym}^d(X)$}

Consider the nonsingular K\"ahler metric $\cansymx$ on $\symx$ (see Remark \ref{re1}).

\begin{thm}\label{th3}
Let $T\, :\, {\rm Sym}^d(X)\, \longrightarrow\, {\rm Sym}^d(X)$ be any
holomorphic automorphism. Then the pulled back K\"ahler form $T^*\cansymx$ coincides
with $\cansymx$. In particular, $T$ is a isometry for the metric $\cansymx$.
\end{thm}

\begin{proof}
Since $\varphi$ (constructed in \eqref{vp}) is the Albanese map for $\text{Sym}^d(X)$, there
is a holomorphic automorphism
$$
\widehat{T}\, :\, \text{Pic}^d(X)\, \longrightarrow\,\text{Pic}^d(X)
$$
such that
\begin{equation}\label{com}
\varphi\circ T\,=\, \widehat{T}\circ\varphi\, .
\end{equation}
{}From \cite{Fa} we know that
$\widehat{T}$ preserves the polarization on $\text{Pic}^d(X)$. A theorem due
to Weil says a holomorphic automorphism of $\mathrm{Jac}(X)\,=\, \text{Pic}^0(X)$ that preserves the
polarization is generated by the following:
\begin{itemize}
\item translations of $\text{Pic}^0(X)$,

\item automorphisms of $\text{Pic}^0(X)$ given by the holomorphic automorphisms of 
$X$, and

\item the inversion of $\text{Pic}^0(X)$ defined by $L\, \longmapsto\, L^*$.
\end{itemize}
(See \cite[Hauptsatz, p.~35]{We}.) But all these three types of automorphisms of 
$\text{Pic}^0(X)$ are isometries for the flat K\"ahler form on $\text{Pic}^0(X)$
constructed in \eqref{m0}. From this it follows immediately that
$\widehat{T}$ is an isometry for the flat K\"ahler form $g_d$ on 
$\text{Pic}^d(X)$ constructed in Section \ref{se3.1}. Since
$\widehat{T}$ is an isometry, from \eqref{com} it follows
immediately that $T^*\omega_d\,=\, \omega_d$.
\end{proof}

\section*{Acknowledgements}

We thank the referee for pointing out a reference. The second-named author wishes 
to thank the University of Hyderabad for hospitality while the work was carried 
out. He is supported by a J. C. Bose Fellowship.

\end{document}